\documentclass[11pt]{article}
\usepackage{amsfonts,epsf,amsmath,amssymb,amsthm}
\usepackage{tikz}
\usepackage[ruled,vlined]{algorithm2e}
\usepackage{graphicx}
\usepackage{latexsym}
\usepackage{enumerate}
\usepackage{setspace}
\usepackage{color}
\usepackage[cp1250]{inputenc}
\newtheorem{thm}{Theorem}[section]

\newtheorem{conj}[thm]{Conjecture}
\newtheorem{cor}[thm]{Corollary}

\newtheorem{lemma}[thm]{Lemma}

\newtheorem{prob}{Problem}

\newcommand{\cF}{{\cal F}}

\begin{document}

\title{On well-covered Cartesian products}

\author{
$^{a}$Bert L. Hartnell
\and
$^{b}$Douglas F. Rall
\and
$^{c}$Kirsti Wash \\
}

\date{\today}

\maketitle

\begin{center}
$^a$ Department of Mathematics and Computing Science, Halifax, Nova Scotia\\
$^c$ Department of Mathematics, Trinity College, Hartford, CT\\
$^b$ Department of Mathematics, Furman University, Greenville, SC\\

\end{center}

\begin{abstract}
In 1970, Plummer defined a well-covered graph to be a graph $G$ in which all maximal independent sets are in fact
maximum. Later Hartnell and Rall showed that if the Cartesian product $G \Box H$ is well-covered, then at least one of $G$
or $H$ is well-covered. In this paper, we consider the problem of classifying all well-covered Cartesian products. In particular,
we show that if the Cartesian product of two nontrivial, connected graphs of girth at least $4$ is well-covered, then
at least one of the graphs is $K_2$.  Moreover, we show that $K_2 \Box K_2$ and $C_5 \Box K_2$ are the only well-covered
Cartesian products of nontrivial, connected graphs of girth at least $5$.
\end{abstract}



\section{Introduction}
A graph $G$ is called well-covered if all maximal independent sets of $G$ have the
same cardinality.   This class of graphs  was introduced by Plummer~\cite{P-1970} in 1970.
To date the study of the class of well-covered graphs seems to be primarily concentrated on
finding good characterizations of various subclasses.  Examples of subclasses of well-covered graphs that have been  characterized
include those of girth 8 or more~\cite{FH-1983}, those of girth at least $5$~\cite{FHN-1993},
those with neither $4$-cycles nor $5$-cycles~\cite{FHN-1994}, and subcubic~\cite{CER-1993}.
See the survey articles by Plummer~\cite{P-1993} and Hartnell~\cite{H-1999}.

Topp and Volkmann~\cite{TV-1992} initiated the study of well-covered graphs in the realm of several
of the standard graph products.  They proved that a lexicographic product $G[H]$ is well-covered if and only
if both of $G$ and $H$ are well-covered.  Although they did not characterize the direct products that
are well-covered, they did prove that if $G$ and $H$ have no isolated vertices and the direct product $G \times H$ is
well-covered, then both of $G$ and $H$ are well-covered and $\alpha(G)\cdot|V(H)|=\alpha(H)\cdot|V(G)|$. (Here $\alpha$
is the vertex independence number.)
Cartesian products turn out to be more difficult to deal with as far as well-covered is concerned. Topp and Volkmann showed
that the Cartesian product of any two complete graphs is well-covered and the Cartesian product of two cycles
is well-covered if and only if at least one of the cycles is $C_3$.  Left unanswered in~\cite{TV-1992}
was the question of whether a Cartesian product $G \Box H$ being well-covered implies that at least
one of $G$ or $H$ is well-covered.

Fradkin~\cite{F-2009} pursued this question and showed that the Cartesian product of any two triangle-free
graphs, neither of which is well-covered, is also not well-covered.  Hartnell and Rall~\cite{HR-2013} settled the
problem of Topp and Volkmann.
\begin{thm} [\cite{HR-2013}] \label{thm:CartesianWC}
If $G$ and $H$ are graphs such that $G \Box H$ is well-covered, then at least one
of $G$ or $H$ is well-covered.
\end{thm}

This suggests an interesting problem.
\begin{prob}\label{prob:1}
For a given graph $H$, characterize those graphs $G$ such that $G \Box H$ is well-covered.
\end{prob}

In Section~\ref{sec:prelim} we give two constructions that suggest the general solution to this problem will
likely be difficult.  In Section~\ref{sec:reduction-of-girth} we provide a partial solution to Problem~\ref{prob:1}
by first considering the Cartesian product of two nontrivial,
connected graphs of girth at least $4$ and show that if this Cartesian product is well-covered, then one of the
two factors is $K_2$.   Moreover, we show that for any two, nontrivial,
connected graphs $G$ and $H$, both with girth at least $5$, the Cartesian product $G\Box H$ is well-covered if and only if
$G\Box H \in \{K_2 \Box K_2, C_5 \Box K_2\}$.

\section{Notation and Preliminary Results} \label{sec:prelim}

A subset $A$ of the vertex set of a graph is \emph{independent} if the vertices in $A$ are pairwise nonadjacent.  If
an independent set $I$ of $V(G)$ is maximal (with respect to being) independent, then every vertex in $V(G)-I$ has at least one
neighbor in $I$.  Thus, a maximal independent set $I$ in $G$ is also dominating.  That is, $N[I]=V(G)$.  The
smallest cardinality, $i(G)$, of a maximal independent set in $G$ is the \emph{independent domination number} of $G$.
The \emph{vertex independence number} of $G$ is the largest cardinality of the maximal independent sets of $G$ and is
denoted $\alpha(G)$.  A graph $G$ is \emph{well-covered} if all the maximal independent sets of $G$ have the
same cardinality.  Thus, $G$ is  well-covered if and only if $i(G)=\alpha(G)$.  If $J$ is any independent
set in a graph $G$, then $J$ can be extended to a maximal independent set of $G$ by repeatedly
adding new vertices that are not dominated by the current set.  Thus, in a well-covered graph $G$
any independent set (in particular, any vertex) is contained in an independent set of cardinality $\alpha(G)$.

If $G$ and $H$ are two finite, simple graphs, then their \emph{Cartesian product}, denoted $G \Box H$,
is the graph with vertex set $V(G) \times V(H)$.  Two vertices $(g_1,h_1)$ and $(g_2,h_2)$ are
adjacent in $G \Box H$ if one of the following holds:
\begin{itemize}
\item $g_1=g_2$ and $h_1h_2 \in E(H)$,
\item $h_1=h_2$ and $g_1g_2 \in E(G)$.
\end{itemize}
The graphs $G$ and $H$ are called the \emph{factors} of $G \Box H$. For a given $g \in V(G)$, the subgraph of $G \Box H$
induced by the set of vertices $\{(g,h)\colon h \in V(H)\}$ is called an \emph{$H$-layer} and is denoted by
$^g\!H$.  In a similar way, for a fixed $h \in V(H)$ the subgraph $G^h$ induced by $\{(g,h)\colon g \in V(G)\}$
is called a \emph{$G$-layer}.  By definition, every $H$-layer is isomorphic to $H$, and every $G$-layer is
isomorphic to $G$.  A significant portion of this paper is dedicated to the study of Cartesian products where one of the factors is $K_2$.
We assume that $V(K_2)=[2]=\{1,2\}$. (For a positive integer $n$ we use $[n]$ to denote the set of
positive integers no larger than $n$.) The Cartesian product $G \Box K_2$ is called the \emph{prism of $G$} and
$G$ is called the \emph{base} of the prism.
In the prism of a graph $G$ we will simplify the notation and write $g^i$ in place of $(g,i)$ for $i \in [2]$ and $g \in V(G)$.
Also for $i \in [2]$, if $A \subseteq V(G)$ we write $A^i$ to denote the set of vertices $\{a^i\colon a\in A\}$.

A vertex of degree 1 in a graph is called a \emph{leaf} and its unique neighbor is called a \emph{support vertex}.
An edge incident with a leaf is a \emph{pendant} edge.
A support vertex that has more than one leaf as a neighbor is called a \emph{strong} support vertex.  If a
graph $G$ has a strong  support vertex $x$, then $G$ is not well-covered.  This is easily seen by letting $I$ be a maximal independent set
of $G-N[x]$.  The set $I \cup \{x\}$ is a maximal independent set of $G$, and yet if $L$ denotes the set of leaves
adjacent to $x$, then $I \cup L$ is a larger independent set.  The \emph{girth}
of a graph $G$, denoted by $g(G)$, is the length of its shortest cycle unless $G$ is a forest, in which case we
define the girth to be $\infty$.   For $u,v \in V(G)$, the length of
the shortest path in $G$ joining $u$ to $v$ is denoted $d_G(u,v)$, or simply $d(u,v)$ if the graph is clear from
the context.

If $I$ is any independent set in a well-covered graph $G$ and $J_1$ and $J_2$ are maximal independent sets in
the induced subgraph $G-N[I]$, then both $I \cup J_1$ and $I \cup J_2$ are maximal independent in $G$.  It
follows immediately that $|J_1|=|J_2|$, and thus $G-N[I]$ is well-covered.  This important and useful property of
a well-covered graph was  proved by Finbow et al.~\cite{FHN-1993}.
We state its contrapositive form since we will often use it  to  show that a graph is not well-covered.
\begin{lemma} [\cite{FHN-1993}] \label{lem:basic}
If $G$ is a graph and $I$ is an independent set of $G$ such that $G-N[I]$ is not
well-covered, then $G$ is not well-covered.
\end{lemma}

A vertex $w$ in an arbitrary graph $H$ is \emph{isolatable} in $H$ if there exists an independent set $M$ in $H$ such
that $\{w\}$ is a (isolated) component in $H-N[M]$.  Equivalently, $w$ is isolatable in $H$ if there exists
an independent set $J$ of $H$ such that $V(H)-N[J]=\{w\}$.
As in~\cite{FHN-1993}, we say a vertex $x$ of a well-covered graph $G$ is \emph{extendable} in $G$
if $G-x$ is well-covered and $\alpha(G-x)=\alpha(G)$. A well-covered graph $G$ is {\it $1$-well-covered} if every $v \in V(G)$ is
extendable in $G$.  In~\cite{FHN-1993} it is proved that in a well-covered graph the notions of extendable and not
isolatable are equivalent.
\begin{thm} [\cite{FHN-1993}] \label{thm:extnotisolatable}
Let $G$ be a well-covered graph.  A vertex $x$ of $G$ is an extendable vertex of $G$ if and only if $x$ is not
isolatable in $G$.
\end{thm}

The next result by Finbow and Hartnell~\cite{FH-2010} shows that graphs of girth at least 4 having no
isolatable vertices are actually well-covered.  We provide a slightly expanded proof for the sake of
completeness.
\begin{thm} [\cite{FH-2010}] \label{thm:noisolatable}
Let $G$ be a graph with $g(G) \ge 4$.  If no vertex in $G$ is isolatable, then  $G$ is well-covered.
\end{thm}
\begin{proof} Suppose that $G$ is a graph of girth at least 4 in which no vertex is isolatable, but that $G$ is not
well-covered.  Consider the set $\cF$ of certain pairs of maximal independent sets of $G$ defined by
\[\cF=\{(L,K)\,\colon \, K\,\, \text{and}\,\, L\,\, \text{are maximal independent in}\,\, G\,\, \text{and}\,\, |K|<|L|=\alpha(G)\}\,.\]
Since $G$ is not well-covered, $\cF\neq \emptyset$.  Let $n=\min\{\,|K-L|\,\colon\, (L,K) \in \cF\}$ and let
$m=\min\{\,|L \cap N(w)| \,\colon\, (L,K)\in \cF,\, n=|K-L|\,\text{and}\,\,w\in K-L\,\}$.  Choose a pair $(L_0,K_0)$
from $\cF$ such that $n=|K_0-L_0|$ and a vertex $v\in K_0-L_0$ such that $m=|L_0 \cap N(v)|$.  Let $u \in L_0 \cap N(v)$ and
set $J=L_0-\{u\}$.  Since $u$ is not dominated by the independent set $J$ and since $u$ is not isolatable ($G$ has no isolatable
vertices by assumption), there is a vertex $x \in N(u)-N(J)$. Since $g(G) \ge 4$, $xv \not\in E(G)$.  Let
$M=J \cup \{x\}$.  We see that $M$ is a maximum independent set of $G$, and thus $(M,K_0) \in \cF$.  If $x\not\in K_0$, then
$|K_0-M|=n$ but $|M \cap N(v)|=m-1$, which is a contradiction to the choice of $m$.  On the other hand, if $x \in K_0$, then $|K_0-M|=n-1$,
which is a contradiction to the choice of $n$. Consequently, $\cF$ is empty, and $G$ is well-covered.
\end{proof}

Finbow et al.~\cite{FHN-1993} were able to give a complete description of the well-covered graphs of girth
at least $5$.
By analyzing this collection Pinter~\cite{P-1997b} observed the following characterization of connected $1$-well-covered
graphs of girth at least $5$.

\begin{thm} [\cite{P-1997b}] \label{thm:girth5noisolates}
If $G$ is a nontrivial, connected well-covered graph of girth at least $5$ that has
no isolatable vertex, then $G\cong K_2$ or $G \cong C_5$.
\end{thm}

To see that a general solution to Problem~\ref{prob:1} is most likely very difficult, consider the following. Let $H$ be a
graph with maximum degree $k$, let $n >k$, and let $I$ be any maximal independent set
of $K_n \Box H$.  If $I$ does not intersect some $K_n$-layer, say $K_n^h$, then the vertices of $K_n^h$ are dominated by
$\cup_{w \in N(h)}(I \cap V(K_n^w))$.  Since $n>k$, the Pigeonhole Principle implies that $|I \cap V(K_n^w)| \ge 2$
for some $w \in N(h)$.  This contradicts the independence of $I$, and so $I$ contains exactly one vertex from
each $K_n$-layer.  Hence, $|I|=|V(H)|$ and we see that $K_n \Box H$ is well-covered.

In fact, by generalizing this idea we can, for any graph $H$ of maximum degree $k$, construct many graphs $G$
that have girth 3  such that $G \Box H$ is well-covered.
Here is one family of such graphs.  Let $r$ be any positive integer.  Let $G$ be any graph whose vertex set $V(G)$ can
be partitioned as $V_1\cup \cdots \cup V_r$ in such a way  that
\begin{enumerate}
\item The subgraph $G[V_i]$ of $G$ is a complete graph of order at least $3$ for each $i \in [r]$, and
\item For each $j \in [r]$ there is a subset $W_j$ of $V_j$ such that $|W_j| \ge k+1$ and
      $N(W_j) \subseteq V_j$.
\end{enumerate}
It is clear that such a graph $G$ is well-covered and $\alpha(G)=r$.
We claim that $G \Box H$ is a well-covered graph with independence number $r|V(H)|$.  For each $i \in [r]$ let
$G_i=G[V_i]$.  Any subset of $V(G)$ formed by choosing one vertex from each of $W_1,\ldots,W_r$ is independent,
and is in fact maximal independent, in $G$.  In addition, for each $h \in V(H)$ we can choose a subset $M_h$
of $W_1 \cup \cdots \cup W_r$ that contains
one vertex from $W_i$ for each $i \in [r]$ in such a way that $M_h \cap M_{h'}=\emptyset$ for any
$h'\in N(h)$.  This last sentence is true since $\Delta(H)=k$ and $|W_i| \ge k+1$.  It now follows that
the set $\cup_{h\in V(H)}(M_h \times \{h\})$ is a maximal independent set of $G \Box H$ with cardinality $r|V(H)|$.  On the
other hand, let $J$ be any maximal independent set of $G \Box H$.  Suppose there exists a $G$-layer, say $G^h$,
such that $|J \cap V(G^h)| < r$.  This implies that there exists $j \in [r]$ such that $J \cap (V_j \times \{h\})=\emptyset$.
In this case no vertex of $W_j \times \{h\}$ is dominated by $J \cap V(G^h)$.  Since $|W_j \times \{h\}| > k=\Delta(H)$,
we arrive at a contradiction.  Hence, $|J \cap V(G^x)| \ge r$, for every $x \in V(H)$.  Since $V(G)$ is covered by
$r$ complete subgraphs we conclude that $|J \cap V(G^x)| = r$ for every $x \in V(H)$.  Consequently, $|J|=r|V(H)|$,
and this implies that $G \Box H$ is well-covered.
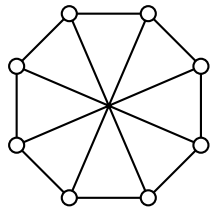
\begin{figure}[htb]
\tikzstyle{every node}=[circle, draw, fill=black!0, inner sep=0pt,minimum width=.2cm]
\begin{center}
\begin{tikzpicture}[thick,scale=.7]
  \draw(0,0) { 
    +(2.50,3.50) -- +(3.50,2.50)
    +(3.50,2.50) -- +(3.50,1.00)
    +(3.50,1.00) -- +(2.50,0.00)
    +(2.50,0.00) -- +(1.00,0.00)
    +(1.00,0.00) -- +(0.00,1.00)
    +(0.00,1.00) -- +(0.00,2.50)
    +(0.00,2.50) -- +(1.00,3.50)
    +(1.00,3.50) -- +(2.50,3.50)
    +(1.00,3.50) -- +(2.50,0.00)
    +(1.00,0.00) -- +(2.50,3.50)
    +(0.00,1.00) -- +(3.50,2.50)
    +(0.00,2.50) -- +(3.50,1.00)
    +(1.00,3.50) node{}
    +(2.50,3.50) node{}
    +(3.50,2.50) node{}
    +(3.50,1.00) node{}
    +(2.50,0.00) node{}
    +(1.00,0.00) node{}
    +(0.00,1.00) node{}
    +(0.00,2.50) node{}

      };
\end{tikzpicture}
\end{center}
\vskip -0.6 cm \caption{The graph $H$.} \label{fig:H}
\end{figure}

For a specific example, let $H$ be the graph shown in Figure~\ref{fig:H}.  Let $G_1,
G_2$ and $G_3$ be complete graphs of order $10$ with $V(G_1)=\{x_1,\ldots,x_{10}\}$,
$V(G_2)=\{y_1,\ldots,y_{10}\}$, and $V(G_3)=\{z_1,\ldots,z_{10}\}$.  The graph $G$ is
formed from the disjoint union of $G_1,G_2,G_3$ by adding any subset of edges among
vertices in the set $\cup_{i=1}^6 \{x_i,y_i,z_i\}$.

\section{Factors of Girth at Least 4} \label{sec:reduction-of-girth}

If a graph $G$ is not connected and has components $C_1, \ldots,C_m$, then for any graph $H$,
the product $G \Box H$ is the disjoint union of  $C_1\Box H, \ldots, C_m \Box H$.  This graph, $G \Box H$,
is well-covered if and only if $C_i \Box H$ is well-covered for every $i \in [m]$.  Consequently,
to determine which Cartesian products are well-covered we can restrict our attention to those in which both
factors are  connected.   As the examples
presented at the end of Section~\ref{sec:prelim} show, the characterization of well-covered Cartesian products if
at least one of the factors has a triangle is unlikely. So we focus on Cartesian products of connected factors
that have girth larger than $3$.

In this section we will prove the following characterization of well-covered Cartesian products of two connected graphs
that each have girth at least 5.
\begin{thm} \label{thm:girths5}
Let $G$ and $H$ be nontrivial, connected graphs, both of which have girth at least $5$.  The Cartesian product
$G \Box H$ is well-covered if and only if  $G \Box H \cong K_2 \Box K_2$ or  $G \Box H \cong C_5 \Box K_2$.
\end{thm}

We first prove through a series of reductions that if
$G$ and $H$ are both nontrivial, connected, triangle-free graphs
such that $G \Box H$ is well-covered, then $G \Box H$ is a prism.

\begin{lemma} \label{lem:girths4}
Let $G$ and $H$ be connected graphs, both of which have order at least $3$ and girth at least
$4$.  If either $G$ or $H$ has an isolatable vertex having degree at least $2$, then $G \Box H$ is not well-covered.
\end{lemma}
\proof
Assume without loss of generality that $G$ has an isolatable vertex $x$ and let
$y_1$ and $y_2$ be distinct neighbors of $x$.  Let $I$ be an independent set in $G$ such that $G-N[I]=\{x\}$.  Since $H$ has order
at least $3$, we fix a vertex $s$ in $H$ with distinct neighbors $t_1$ and $t_2$.  We assume without loss of generality
that $\deg_H(t_1) \le \deg_H(t_2)$.  If $\deg_H(t_1)=1= \deg_H(t_2)$, then let $J=I \times \{t_1,t_2\}$.  If
$1=\deg_H(t_1)< \deg_H(t_2)$, then let $J=(I \times \{t_1,t_2\}) \cup (\{y_1\} \times (N_H(t_2)-\{s\}))$.  Finally,
suppose that $1 < \deg_H(t_1)$.  Let  $T=N_H(\{t_1,t_2\})-\{s\}$, let  $A=T-N_H(t_2)$, let $B=T-A$, and let
$J=(I \times \{t_1,t_2\}) \cup (\{y_1\} \times A) \cup ( \{y_2\} \times B)$.

Since $H$ is triangle-free and because of the definition of $I$, we see that all three cases above $J$ is
independent in $G \Box H$.  In addition, the vertex $(x,s)$ is a strong support vertex adjacent to leaves
$(x,t_1)$ and $(x,t_2)$ in $G \Box H -N[J]$. Therefore, by Lemma~\ref{lem:basic} it follows
that $G \Box H$ is not well-covered.  \qed

The following corollary follows immediately from Lemma~\ref{lem:girths4}, Theorem~\ref{thm:extnotisolatable},
and Theorem~\ref{thm:noisolatable}.

\begin{cor} \label{cor:1-wellcoveredfactors}
Let $G$ and $H$ be connected graphs, both of which have  minimum degree
at least $2$ and girth at least $4$.  If $G \Box H$ is well-covered, then
both $G$ and $H$ are $1$-well-covered.
\end{cor}

Since a leaf in a graph of order at least 3 is isolatable,
we are now able to weaken the hypothesis of Lemma~\ref{lem:girths4} to cover the case when one of the
factors has an isolatable vertex of degree 1 while making no assumption about the order of the other factor.

\begin{lemma} \label{lem:girth4isolatable}
Let $G$ and $H$ be nontrivial, connected graphs both having girth at least $4$.  If $G$ has minimum degree $1$ and order at least $3$,
then $G \Box H$ is not well-covered.
\end{lemma}
\proof
Assume that $x$ is a leaf in $G$ that is adjacent to a support vertex $y$, and assume that $N_G(y)=\{x,z_1,\ldots,z_r\}$
where $r\ge 1$.  Let $I=\{z_1,\ldots,z_r\}$.  Note that since $g(G) \ge 4$, $I$ is independent in $G$ and $x$ is isolated in $G-N[I]$.
Fix any vertex $s$ of $H$ and assume that $N_H(s)=\{t_1,\ldots,t_m\}$.  If $\deg_H(t_1)=1$, then let $A=\emptyset$; otherwise, let
$A=N_H(t_1)-\{s\}$.  Let $J=(I \times N_H(s)) \cup (\{y\} \times A)$.  Since both $G$ and $H$ have girth at least $4$, it follows
that $J$ is independent in $G \Box H$.  The vertex $(x,s)$ is a strong support vertex in $G \Box H -N[J]$, adjacent to leaves
$(y,s)$ and $(x,t_1)$.  By Lemma~\ref{lem:basic} we conclude that $G\Box H$ is not well-covered.  \qed

We now make the final reduction in the case when both nontrivial, connected factors have girth at least 4.  Here we do not
assume that one of the factors has an isolatable vertex.

\begin{lemma}  \label{lem:girthge4-order3}
If $G$ and $H$ are connected graphs both having order at least $3$ and girth at least $4$, then $G \Box H$ is not well-covered.
\end{lemma}
\proof
Assume that $G$ and $H$ are both of order at least $3$ and girth at least $4$.  Suppose, in order to arrive at a
contradiction, that $G \Box H$ is well-covered.  By Lemma~\ref{lem:girths4} and Lemma~\ref{lem:girth4isolatable}, neither
$G$ nor $H$ has an isolatable vertex, and hence $\delta(G) \ge 2$ and $\delta(H) \ge 2$.  In addition, it follows from
Theorem~\ref{thm:noisolatable} that both $G$ and $H$ are well-covered.  Fix a vertex $s_2$ in $H$, and distinguish one of
its neighbors, say $s_1$.  Let $N_H(s_2)=\{s_1,t_1,\ldots,t_m\}$ for some $m \ge 1$ and
let $N_H(s_1)=\{s_2,z_1,\ldots,z_r\}$ for some $r \ge 1$.  Select a vertex $y$ in $G$,
and let $I$ be a maximal independent set of $G-N[y]$.  Since $y$ is not isolatable in $G$, the graph $G-N[I]$ contains at
least one neighbor $x$ of $y$.  Furthermore, since $g(G) \ge 4$ and $G$ is well-covered, it follows from Lemma~\ref{lem:basic}
that $G -N[I]$ contains exactly the  two adjacent vertices $x$ and $y$.  Let $w \in N_G(x)-\{y\}$ and let
\[J=(I \times \{s_1\}) \cup (N_G(y) \times \{t_1,\ldots,t_m\}) \cup (\{w\} \times \{z_1,\ldots,z_r\})\,.\]
We note that $I$ was chosen to be independent; $N_G(y) \times \{t_1,\ldots,t_m\}$ is independent since both $G$ and $H$
are triangle-free; and $\{w\} \times \{z_1,\ldots,z_r\}$ is independent since $g(H) \ge 4$.  As a result, $J$ is independent by
the definition of the edge set of $G \Box H$.  However, by appealing to Lemma~\ref{lem:basic}, we now arrive at a contradiction
since $G \Box H-N[J]$ has a strong support vertex $(y,s_1)$ that is adjacent to the leaves $(y,s_2)$ and $(x,s_1)$.
Consequently, $G \Box H$ is  not well-covered. \qed

\begin{cor} \label{cor:mustbeprism}
If $G$ and $H$ are nontrivial, connected graphs with girth at least $4$ such that
$G \Box H$ is well-covered, then at least one of $G$ or $H$ is the graph $K_2$.
\end{cor}

Because of Corollary~\ref{cor:mustbeprism}, for the remainder of the paper we restrict ourselves
to prisms of graphs that have girth at least 4.  The following corollary, concerning the prism of certain graphs, follows directly
from Lemma~\ref{lem:girth4isolatable}.

\begin{cor} \label{cor:leaves}
If a connected graph $G$ has girth at least 4, has order at least 3 and has minimum degree 1, then $G\Box K_2$
is not well-covered.
\end{cor}

Recall from Theorem~\ref{thm:noisolatable} that a graph is well-covered if it has girth at least $4$ and has no isolatable vertices. We next ask
whether a graph $G$ with an isolatable vertex can have a well-covered prism. The next result shows the answer is no if
$G$ has girth at least $5$.

\begin{thm} \label{thm:girth5isolatable}
If $G$ has girth at least $5$ and has an isolatable vertex, then $G\Box K_2$ is not well-covered.
\end{thm}

\proof
By Corollary~\ref{cor:leaves} we may assume $G$ has  minimum degree at least $2$. Let $x$ be an isolatable vertex of $G$,
let $N(x) = \{y_1, \dots, y_k\}$, and let $A_i$ represent the vertices other than $x$ that are adjacent to $y_i$ for $i \in [k]$. Since
$G$ has girth at least 5, $N(x)$ is independent and the sets $A_1,\ldots,A_k$ are independent and pairwise disjoint.

Suppose first that one of the sets $A_1, \ldots,A_k$, say $A_1$, has cardinality 1. Let $A_1 = \{a\}$.
Note by assumption that there exists an independent set $J$ that
isolates $x$ in $G - N[J]$ and necessarily it must contain $a$. This implies that $y_1^1$ and $x^2$ are leaves
adjacent to the same support vertex in $G \Box K_2 - N[J^2]$.  We conclude that in this case $G\Box K_2$ is not well-covered.

So we may assume that $|A_i| \ge 2$ for all $i \in[k]$. Let $I$ be a maximal independent set of $G - N[\{y_1, \dots, y_k\}]$.

\begin{enumerate}
\item[(a)] Suppose first that $I$ dominates at least one of the sets $A_1,\ldots,A_k$, say $A_1$.
Let $J$  be an independent set that isolates $x$ in $G - N[J]$ and
let $S=I^1 \cup \left( \cup_{i\in [k]}(J \cap A_i) \right)^2$.
It follows that $S$ is independent in $G \Box K_2$, and  $y_1^1$ and $x^2$ are leaves in $G \Box K_2 - N[S]$ both adjacent to $x^1$.
Again this implies that $G\Box K_2$ is not well-covered.

\item[(b)] Now suppose that there exists $i\in[k]$ such that $I$ does not dominate at least two vertices of $A_i$.
Without loss of generality we may assume $i=1$ and $a_1$ and $b_1$ are vertices in $A_1$ not dominated by $I$.
Suppose first that $\deg_G(a_1) >2$
and let $w \in N(b_1) - \{y_1\}$. Thus, $w$ can be adjacent to only one vertex of $N(a_1) - \{y_1\}$ for otherwise $G$ contains
a $4$-cycle. In this case, choose $z \in N(a_1) - \{y_1\}$ that is not adjacent to $w$ and choose
$S = I^1 \cup \{y_2^1, \dots, y_k^1, w^2, z^2\}$ so that $a_1^1$ and $b_1^1$ are leaves in $G \Box K_2 - N[S]$. A similar argument
can be used when $\deg_G(b_1) > 2$.

Next, suppose that $\deg_G(a_1) = 2$ and $\deg_G(b_1) = 2$ where $N(a_1) = \{y_1, z\}$ and $N(b_1) = \{y_1, w\}$.
If $wz \not\in E(G)$, then again choose $S = I^1 \cup \{y_2^1, \dots, y_k^1, w^2, z^2\}$ so that $a_1^1$ and $b_1^1$ are leaves
in $G \Box K_2 - N[S]$. So we may assume that $wz \in E(G)$.

Suppose first that $z \not\in \cup_{i=2}^k A_i$. Let $J$ be any independent set in $G$ that isolates $x$ in $G - N[J]$.
Let $M = (J-\{z\}) - (J\cap A_1)$
 and choose $S = M^1 \cup \{z^2, y_2^2\}$. One can easily verify that $S$ is indeed an independent set in $G \Box K_2$. We claim
 that $a_1^1$ and $x^1$ are leaves in $G\Box K_2 - N[S]$. Note that $N(a_1^1) = \{a_1^2, y_1^1, z^1\}$ and
 $N(x^1) = \{x^2, y_1^1, \dots, y_k^1\}$. The vertices $z^1$ and $a_1^2$ are dominated by $z^2$ so $a_1^1$ is a leaf, $M^1$ dominates
 $y_2^1, \dots, y_k^1$, and $y_2^2$ dominates $x^2$. Thus, $x^1$ is also a leaf. A similar argument works if $w \not\in \cup_{i=2}^k A_i$.

So we may assume that $z \in A_i$ and $w \in A_j$ for some $i,j \in \{2, \dots, k\}$. It follows that $i \ne j$ since $A_i$ is
an independent set. If every independent set that isolates $x$ contains $z$, then none of these independent sets contain $w$ as
$wz \in E(G)$. In this case, let $J$ be such an independent set that doesn't contain $w$, and let $M = J - (J \cap A_1)$. Let
$S = M^1 \cup \{w^2, y_i^2\}$. Note that $S$ is independent since $w \not\in J \cup A_i$. We claim that $b_1^1$ and $x^1$ are leaves in
$G \Box K_2 - N[S]$. Note that $N(b_1^1) = \{b_1^2, w^1, y_1^1\}$ and $N(x^1) = \{x^2, y_1^1, \dots, y_k^1\}$. The vertex $w^2$ dominates
$b_1^2$ and $w^1$ so $b_1^1$ is a leaf. Since $M^1$ dominates $\{y_2^1, \dots, y_k^1\}$ and $y_i^2$ dominates $x^2$, it
follows that $x^1$ is a leaf.
On the other hand, if there exists an independent set that isolates $x$ and does not contain $z$, we may choose a set $S$ so that
$a_1^1$ and $x^1$ are leaves in $G \Box K_2 - N[S]$.  In each of these cases the removal of the closed neighborhood of an
independent set from $G \Box K_2$ created a strong support, and hence $G \Box K_2$ is not well-covered by Lemma~\ref{lem:basic}.

\item[(c)] Finally, suppose that $I$ does not dominate exactly one vertex from each $A_i$ for $i \in [k]$, say $a_i$. Note that if
$\deg_G(x) \ge 3$, then there exist $i$ and $j$ such that $a_ia_j \not\in E(G)$ for otherwise $G$ would contain a triangle. Without
loss of generality, we may assume $i = 1$ and $j = 2$. In this case, choose $S = I^1 \cup \{a_1^2, a_2^2\}$ and note that $y_1^1$ and
$y_2^1$ are leaves, both adjacent to the same support vertex in $G\Box K_2 - N[S]$.
If $\deg_G(x) = 2$ and $a_1a_2 \not\in E(G)$, then the above set $S$ still works. So we
may assume that $a_1a_2 \in  E(G)$. Choose $S = I^1 \cup \{a_2^1, a_1^2, t^2\}$ where $t \in A_2 - \{a_2\}$. One can verify that
$S$ is an independent set. We claim that $y_1^1$ and $x^2$ are leaves, both adjacent to $x^1$ in $G\Box K_2 - N[S]$.
Note that $N(y_1^1) = \{x^1, y_1^2\} \cup A_1^1$ and $N(x^2) = \{x^1, y_1^2,y_2^2\}$. The set $I^1 \cup \{a_2^1\}$ dominates
all of $A_1^1$ and $\{a_1^2, t^2\}$ dominates $y_1^2$ and $y_2^2$.
Thus, $y_1^1$ and $x^2$ are indeed leaves.  Hence, by Lemma~\ref{lem:basic} $G \Box K_2$ is not well-covered.
\end{enumerate}
Having considered all cases, we may conclude that $G \Box K_2$ is not well-covered.  \qed

It is straightforward to verify that both of the prisms $K_2 \Box K_2$ and $C_5 \Box K_2$ are well-covered.  This together
with Theorems~\ref{thm:noisolatable}, \ref{thm:girth5noisolates}, \ref{thm:girth5isolatable} proves Theorem~\ref{thm:girths5},
which we restate here for completeness.

\noindent \textbf{Theorem~\ref{thm:girths5}}. \emph{
Let $G$ and $H$ be nontrivial, connected graphs, both of which have girth at least $5$.  The Cartesian product
$G \Box H$ is well-covered if and only if  $G \Box H \cong K_2 \Box K_2$ or  $G \Box H \cong C_5 \Box K_2$.
}

We now show that, regardless of girth, if the base graph has no isolatable vertices,
then its prism is well-covered.  Of course by the characterization theorems of Finbow et al.~every well-covered graph
of girth larger than $5$ has an isolatable vertex.
\begin{thm} \label{thm:wcbasenoisolatables}
If $G$ is a well-covered graph with no isolatable vertices, then $G \Box K_2$ is well-covered.
\end{thm}
\proof  Let $m=\alpha(G)$ and let $I$ be any maximal independent set of $G \Box K_2$.  Suppose that
$|I \cap V(G^1)| < m$.  If necessary, enlarge $I \cap V(G^1)$ to an independent set $A$ of $G^1$
such that $|A|=m-1$.  The subgraph $G^1-N[A]$ is a clique of order at least 2, for otherwise $G^1$
would contain an independent set of cardinality at least $m+1$, which is a contradiction.  This
implies that $G^1-N[I \cap V(G^1)]$ contains a clique $C$ with vertices $x_1^1,\ldots,x_k^1$ for some $k\ge 2$.
The vertices in $C$ are dominated by the set $I \cap V(G^2)$, which implies that $x_1^2, x_2^2 \in I$.
Since $C$ is a clique, we have that $x_1x_2 \in E(G)$, which contradicts the assumption that $I$ is
an independent set.  Consequently, $|I \cap V(G^1)| = m$.  Similarly, $|I \cap V(G^2)|=m$.  We conclude
that $|I|=2m$, and hence $G \Box K_2$ is well-covered.  \qed

In particular, if $g(G)=4$ and $G$ has no isolatable vertex, then by Theorem~\ref{thm:noisolatable}
it follows that $G$ is well-covered.  This proves the following corollary to Theorem~\ref{thm:wcbasenoisolatables}.

\begin{cor} \label{cor:girth4noisolatables}
If $G$ has girth $4$ and has no isolatable vertex, then the prism of $G$ is well-covered.
\end{cor}

There are many examples of $1$-well-covered graphs of girth $4$ as in the hypothesis of Corollary~\ref{cor:girth4noisolatables}.
One such graph, commonly known as $WL_8$, is shown in Figure~\ref{fig:WL8}.

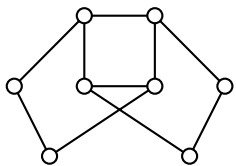
\begin{figure}[htb]
\tikzstyle{every node}=[circle, draw, fill=black!0, inner sep=0pt,minimum width=.2cm]
\begin{center}
\begin{tikzpicture}[thick,scale=.7]
  \draw(0,0) { 
    +(1.33,2.67) -- +(2.67,2.67)
    +(2.67,2.67) -- +(2.67,1.33)
    +(2.67,1.33) -- +(1.33,1.33)
    +(1.33,1.33) -- +(1.33,2.67)
    +(1.33,2.67) -- +(0.00,1.33)
    +(0.00,1.33) -- +(0.67,0.00)
    +(0.67,0.00) -- +(2.67,1.33)
    +(1.33,1.33) -- +(3.33,0.00)
    +(3.33,0.00) -- +(4.00,1.33)
    +(4.00,1.33) -- +(2.67,2.67)
    +(1.33,2.67) node{}
    +(2.67,2.67) node{}
    +(1.33,1.33) node{}
    +(2.67,1.33) node{}
    +(0.67,0.00) node{}
    +(3.33,0.00) node{}
    +(0.00,1.33) node{}
    +(4.00,1.33) node{}
    };
\end{tikzpicture}
\end{center}
\vskip -0.6 cm \caption{The graph $WL_8$.} \label{fig:WL8}
\end{figure}


\section{Summary}
We have shown that if a Cartesian product of two nontrivial, connected, triangle-free graphs is well-covered,
then this Cartesian product is a prism, say $G \Box K_2$.  In addition, if $G$ has girth at least $5$, then $G$ is either $K_2$
or $C_5$, and indeed the prisms $K_2\Box K_2$ and $C_5 \Box K_2$ are both well-covered.
If the girth of $G$ is $4$ and
$G$ has no isolatable vertex, then the prism $G \Box K_2$ is well-covered.  We suspect that being well-covered and having no
isolatable vertex is also a necessary condition for the prism of a graph of girth $4$ to be well-covered.
We end by stating this as a conjecture.

\begin{conj}
Let $G$ be a connected, triangle-free graph that contains a cycle of order $4$.  If $G \Box K_2$ is well-covered,
then $G$ has no isolatable vertex.
\end{conj}

\section*{Acknowledgements}
The second author is supported by a grant from the Simons Foundation (Grant Number 209654 to Douglas F. Rall).

\end{document}